\documentclass[12pt, leqno]{amsart}
\usepackage[OT2,T1]{fontenc} 
\DeclareSymbolFont{cyrletters}{OT2}{wncyr}{m}{n}
\DeclareMathSymbol{\Sha}{\mathalpha}{cyrletters}{"58}

\usepackage{indentfirst}
\usepackage{amstext}
\usepackage{amsopn}
\usepackage{amsfonts}
\usepackage{amsmath}
\usepackage{latexsym}
\usepackage{amscd}
\usepackage{amssymb}
\usepackage{amsmath}
\usepackage[all,cmtip]{xy}
\usepackage{leftidx}
\usepackage{graphicx}
\usepackage{tikz}
\usepackage{ulem}
\usepackage{hyperref}

=\oddsidemargin \font\teneufm=eufm10 \font\seveneufm=eufm7
\font\fiveeufm=eufm5
\newfam\eufmfam
\textfont\eufmfam=\teneufm \scriptfont\eufmfam=\seveneufm
\scriptscriptfont\eufmfam=\fiveeufm

\let\goth\mathfrak

\def\cB{\mathcal B}

\def\cO{\mathcal O}

\def\cE{\mathcal E}

\def\cP{\mathcal P}
\def\cp{\mathfrak{p}}
\def\cU{\mathcal U}

\def\GG{\mathbb{G}}

\def\WW{\mathbf{W}}
\def\VV{\mathbf{V}}

\def\gB{\goth B}
\def\gG{\goth G}
\def\gH{\goth H}

\def\gT{\goth T}

\def\gX{\goth X}
\def\gZ{\goth Z}

\def\gp{\goth p}

\def\1{\mbox{\bf 1}}

 \DeclareMathOperator{\Hom}{Hom}
\DeclareMathOperator{\Aut}{Aut}

\DeclareMathOperator{\uF}{\underline{F}}
\DeclareMathOperator{\uG}{\underline{G}}

\DeclareMathOperator{\uP}{\underline{P}}
\DeclareMathOperator{\uU}{\underline{U}}
\DeclareMathOperator{\uZ}{\underline{Z}}

\DeclareMathOperator{\Out}{Out}

\DeclareMathOperator{\GL}{\rm GL}

\newcommand{\incl}[1][r]
{\ar@<-0.2pc>@{^(-}[#1] \ar@<+0.2pc>@{-}[#1]}

\newtheorem{stheorem}{Theorem}[section]
\newtheorem{sclaim}[stheorem]{Claim}

\newtheorem{scorollary}[stheorem]{Corollary}
\newtheorem{slemma}[stheorem]{Lemma}
\newtheorem{sproposition}[stheorem]{Proposition}
\newtheorem{sremark}[stheorem]{Remark}
\newtheorem{sremarks}[stheorem]{Remarks}
\newtheorem{sexample}[stheorem]{Example}

\newtheorem{ssetting}[stheorem]{Setting}

\theoremstyle{definition}

\numberwithin{equation}{section}

\def\ZZ{\mathbb{Z}}

\def\PP{\mathbb{P}}

\def\gE{\mathfrak{E}}

\def\gG{\mathfrak{G}}

\def\gQ{\mathfrak{Q}}

\def\gU{\mathfrak{U}}
\def\Par{\mathrm{Par}}

\def\cO{\mathcal{O}}

\def\ol{\overline}

\def\fppf{\text{\rm fppf}}

\def\Lie{\mathop{\rm Lie}\nolimits}

\def\2int{\mathop{2\int}\nolimits}

\def\Spec{\mathop{\rm Spec}\nolimits}

\def\Lie{\mathop{\rm Lie}\nolimits}

\def\Hom{\mathop{\rm Hom}\nolimits}

\def\Stab{\mathop{\rm Stab}\nolimits}

\def\Gal{\mathop{\rm Gal}\nolimits}

\def\Aut{\text{\rm{Aut}}}
\def\Out{\text{\rm{Out}}}
\def\sm{\smallskip}

\def\Par{\text{\rm{Par}}}

\def\resp.{\mathop{\rm resp.}\nolimits}
\def\limproj{\mathop{\oalign{lim\cr
\hidewidth$\longleftarrow$\hidewidth\cr}}}

\def\lgr{\longrightarrow}

\font\math=cmmi10
\def\varpi{\hbox{\math\char'44}}

\def\simlgr{\buildrel\sim\over\lgr}

\def\pa{\S\kern.15em }

\def\un{\uppercase\expandafter{\romannumeral 1}}
\def\deux{\uppercase\expandafter{\romannumeral 2}}
\def\trois{\uppercase\expandafter{\romannumeral 3}}
\def\quatre{\uppercase\expandafter{\romannumeral 4}}
\def\cinq{\uppercase\expandafter{\romannumeral 5}}
\def\six{\uppercase\expandafter{\romannumeral 6}}

\def\hfl#1#2#3{\smash{\mathop{\hbox to#3{\rightarrowfill}}\limits
^{\scriptstyle#1}_{\scriptstyle#2}}}
\def\gfl#1#2#3{\smash{\mathop{\hbox to#3{\leftarrowfill}}\limits
^{\scriptstyle#1}_{\scriptstyle#2}}}

\title[Local-global principle]{A local-global principle for twisted flag varieties }

\author{P. Gille}\address{UMR 5208
Institut Camille Jordan - Universit\'e Claude Bernard Lyon 1
43 boulevard du 11 novembre 1918
69622 Villeurbanne cedex - France 
}

\email{gille@math.univ-lyon1.fr}

\author{R. Parimala}\address{Departement  of Mathematics and Computer Science,
MSC W401, 400 Dowman Dr. Emory University
Atlanta, GA 30322
USA  
}
\email{parimala.raman@emory.edu}

\date{\today}

\begin{document}

 \begin{abstract} We prove a local-global principle
for twisted flag varieties over a semiglobal field. 
  
\smallskip


\end{abstract}

\maketitle


\bigskip

\section{Introduction}\label{section_intro}

Let $T$ be a complete discrete valuation ring with fraction field $K$
and residue field $k$. Let $X$ be a smooth, projective, geometrically 
integral curve over $K$. Let $F=K(X)$ be the function field of $X$ and let $t$ be an uniformizing parameter of $T$.
We prove the following theorem which settles a conjecture of Colliot-Th\'el\`ene, Suresh and the second author for function fields of $p$-adic curves \cite[conjecture 1]{CTPS}, in the  very general context of semiglobal fields.

 \begin{stheorem} \label{thm_patching0}
 (see Th. \ref{thm_patching}) Let $G$ be a reductive 
 $F$--group and assume that $p$ does not divide the order of the automorphism group of the absolute root system of 
 $G_{ad}$.
 Let $Z$ be a twisted flag $F$--variety of $G$.
 Then $Z(F) \not = \emptyset$ if and only if $Z(F_v) \not = \emptyset$
 for all discrete valuations of $F$ arising from models of $X$.
 \end{stheorem}

The result was known for the cases of smooth projective
quadrics and Severi-Brauer varieties \cite[th. 3.1]{CTPS}
and also for generalized Severi-Brauer varieties (Reddy-Suresh \cite[th. 2]{RS}).

If we use all rank one valuations, the
result holds unconditionnally (Cor. \ref{cor_main_hhk3}).
On the other hand, if $G$ extends to a reductive group 
scheme over a smooth model $\gX$ of  $X$, 
the result is unconditional (Cor. \ref{cor_main_hhk2}).
 
 \smallskip
 
Let us review the contents.
Given a henselian couple $(A,I)$ and 
a reductive $A$-group scheme $H$,
section 2 deals with generation of groups $H(A)$ by 
unipotent elements and also with the quotient 
of $H(A)/RH(A)$ by the (normal) subgroup of $R$-trivial elements
defined in \cite{GS}. Section 3 is devoted to mild improvements
of patching techniques of Harbater-Hartmann-Krashen 
involving more analytical functions and leads to various
group decompositions (prop. \ref{prop_R_eq});
this permits to obtain a patching theorem
for twisted flag varieties (th. \ref{thm_key}).
Section 4  deals with the setting
of the beginning of the introduction and 
provides a weaker version of  the main result.
What remains to do is mostly a local study  for
a two dimensional complete regular local ring $R$ with parameters $t,s$
and torsors over its localization $R[t^{-1},s^{-1}]$.
This is achieved in  section 5 by making use of the 
loop torsors over $R[t^{-1},s^{-1}]$ and of the results of 
\cite{Gi2}.

\medskip

\noindent{\bf Acknowledgements.} 
We thank  
Venapally Suresh for a careful reading of the manuscript
and fruitful suggestions. We thank 
  David Harbater for valuable comments on a preliminary version of the paper.  
 Finally we  thank Gabriel  Dospinescu
for telling us about Schneider's book.
 \bigskip  
 
 \smallskip

 \noindent{\bf Conventions and notation.}
$(a)$ A variety $V$ over a field $F$ means
a separated $F$--scheme of finite type which is 
 integral \cite[Tag 020D]{St}.
 
 \sm
 
\noindent $(b)$ Let $G$ be a reductive $F$--algebraic group and let $F_s$ be a separable closure of $F$.
Let $(B,T)$ be a Killing couple of $G_{k_s}$ and denote
$\Delta(G_{F_s})= \Delta(G_{F_s},B,T)$  the associated Dynkin diagram
(there is a canonical bijection $\Delta(G_{F_s},B,T) \simlgr \Delta(G_{F_s},B',T')$
for another choice of Killing couple).
This diagram 
$\Delta(G_{F_s})$  is equipped 
 with the star action of $\Gal(F_s/F)$ \cite[\S 6.4]{BoT65}.
 We recall that there is an order preserving bijection $I \to P_I$ between
 the subsets of $\Delta(G_{F_s})$ and the $F_s$--parabolic subgroups of 
 $G_{F_s}$ containing $B$
 (or equivalently with the $G(F_s)$--conjugacy classes of parabolic subgroups of $G_{F_s}$).
 Since minimal $F$--parabolic subgroups are $G(F)$-conjugated 
 we denote by  $\Delta_0(G) \subset \Delta(G_{F_s})$ be conjugacy class of 
 $P_0 \times_F F_s$ where $P_0$ is a minimal $F$--parabolic subgroup of $G$ (it is stable under the star action).
 The triple $(\Delta(G_{F_s}), \hbox{star action}, \Delta_0(G))$ is called 
 the Tits index of $G$. By abuse of notation, $\Delta_0(G)$ alone is called the Tits index. 
 By a twisted $F$--flag variety of $G$, we mean 
 a $G$-variety $X$ which becomes isomorphic over $F_s$
 to some $G_{F_s}/P_I$ as $G_{F_s}$-variety 
 (such an $I$ is unique).
According to \cite[prop. 1.3]{MPW},  
$I$ is invariant under the star action and 
$X$ is $G$-isomorphic to the variety of $F$--parabolic
subgroups of type $I$ which is denoted in this paper
by $\Par_I(G)$. 
 
 Those  varieties are called sometimes projective 
 homogeneous $G$--varieties (as in \cite{CTPS})
 but we warn the reader of the danger of this terminology
 since there exist in positive characteristic
 $F$--subgroups $Q$ which are not smooth such that 
 $G/Q$ is a projective $F$--variety \cite{W}. 
 In the appendix \ref{app_parabolic}, we show that 
 a homogeneous $G$--variety $X$ is a variety of parabolic subgroups
 if and only if the action is transitive in the sense
 that $G(E)$ acts transitively on $X(E)$ for each $F$ field $E$ (Prop. \ref{prop_homog}).

 \sm
 
\noindent $(c)$ 
We use mainly  the terminology and notation of Grothendieck-Dieudonn\'e \cite[\S 9.4  and 9.6]{EGA1}
 which agrees with that  of Demazure-Grothendieck used in \cite[Exp. I.4]{SGA3}. Let $S$ be a scheme and let $\cE$ be a quasi-coherent sheaf over $S$.
 For each morphism  $f:T \to S$, 
we denote by $\cE_{T}=f^*(\cE)$ the inverse image of $\cE$ 
by the morphism $f$.
 We denote by $\VV(\cE)$ the affine $S$--scheme defined by 
$\VV(\cE)=\Spec\bigl( \mathrm{Sym}^\bullet(\cE)\bigr)$;
it is affine  over $S$ and 
represents the $S$--functor $Y \mapsto \Hom_{\cO_Y}(\cE_{Y}, \cO_Y)$  \cite[9.4.9]{EGA1}. 
  We assume now that $\cE$ is locally free of finite rank and denote by $\cE^\vee$ its dual.
In this case the affine $S$--scheme $\VV(\cE)$ is  of finite presentation
(ibid, 9.4.11); also
the $S$--functor $Y \mapsto H^0(Y, \cE_{Y})= 
\Hom_{\cO_Y}(\cO_Y, \cE_{Y} )$ 
is representable by the  affine $S$--scheme $\VV(\cE^\vee)$
which is also denoted by  $\WW(\cE)$  \cite[I.4.6]{SGA3}.

 \section{Around Kneser-Tits' problem} 
 Let $A$ be a semilocal ring, $I$ its radical. 
 
 \subsection{Subgroups of elementary elements}\label{subsec_PS}
 Let  $H$ be a reductive $A$-group scheme which   is strictly $A$--isotropic, that is, 
 all factors of $H_{ad}$ are isotropic.
Let  $H$ be a strictly proper parabolic $A$--subgroup of $H$ and let $P^{-}$ be 
a $A$--parabolic subgroup opposite to $P$. 
The subgroup  of $H(A)$ generated by $R_u(P)(A)$ and $R_u(P^-)(A)$ does not depend 
of the choice of $P$ and $P^{-}$, this is a normal subgroup of $H(A)$ 
which is denoted by $H(A)^+$ \cite[Thm. 1  and comments]{PS}.

\begin{slemma}\label{lem_GS0}
We have  $H(A)=H(A)^+ \,  P(A)$. 
\end{slemma}

\begin{proof} 
We put $U=R_u(P)$,  $U^{-}=R_u(P^{-})$.
 According to \cite[XXVI.5.2]{SGA3}, we have a decomposition 
 $H(A)  =U(A)  \, U^{-}(A) P(A)$ whence $H(A)  = H(A)^+ \, P(A)$. 
\end{proof}

A special  case of \cite[prop. 7.7]{GS} is the following.

\begin{slemma}\label{lem_GS}
Assume that $(A,I)$ is an henselian couple and that  $H$ is  semisimple simply connected.
Then the map  $H(A)/H(A)^+ \to H(A/I)/ H(A/I)^+$ is an isomorphism. 
\end{slemma}

\subsection{R-equivalence}

We assume now that  $A$ is a nonzero finite $k$--algebra for  a field $k$.
Then $A$ is an Artinian $k$--algebra. We denote by $I$ its Jacobson radical
and remind  the reader that $(A,I)$ is an henselian couple.

Let $H$ be an $A$--reductive group scheme and 
 consider the Weil restriction \break $G=R_{A/k}(H)$, this is a smooth  affine connected algebraic 
 $k$--group \cite[A.5.9]{CGP}. We have $G(k)=H(A)$. 
 We use $R$-equivalence for group schemes over rings as defined in 
 \cite{GS}.

 \begin{slemma}\label{lem_R_eq}
 We decompose $A=A_1 \times \dots \times A_d$ where each $A_i$ is a local artinian 
 $k$--algebra of residue field $k_i$.
  
  \smallskip
  
  \noindent  (1) The map $H(A) \to \prod_i H(k_i)$ induces  isomorphisms
 $$
  G(k)/RG(k) \simlgr   H(A)/RH(A) \simlgr  \prod_i H(k_i)/RH(k_i) .
 $$

  \smallskip
  
  \noindent  (2) If $H$ is furthermore simply connected and strictly isotropic, 
  we have a commutative diagram  of isomorphisms
 
\[\xymatrix{
  H(A)/H(A)^+  \ar[r]^\sim  \ar[d]^\wr &    H(A)/RH(A) \ar[d]^\wr  & \ar[l]^\sim    G(k)/RG(k) \\ 
 \prod_i H(k_i)/H(k_i)^+  \ar[r]^\sim & \prod_i H(k_i)/RH(k_i)  . 
}\]

 \end{slemma}

 \begin{proof} 
 (1) The first isomorphism $G(k)/RG(k) \simlgr H(A)/RH(A)$  is a formal fact \cite[lemma 2.4]{GS}.
 Let $n$ be the smallest positive integer such that $I^n=0$.
 We put $A_j=A/I^{j+1}$ for $j=0, \dots, n-1$, it comes 
 with the ideal $J_j= I^{j}/I^{j+1}=  I \, A_j $ which satisfies $J_j^2=0$. We have 
 $A_j/J_j= A_{j-1}$ for $j=1, \dots, n-1$.
 
 We put $G_j= R_{A_j/k}(H_{A_j})$ for $j=0,\dots , n-1$; 
 we have a sequence of homomorphisms 
 $G=G_{n} \to G_{n-1} \to \dots \to G_1 \to G_0 = \prod_i R_{k_i/k}(H_{k_i})$.
 It induces homomorphisms
 $$
 G(k)/R = G_{n}(k)/R\to G_{n-1}(k)/R \to \dots \to G_1(k)/R \to G_0(k)/R = \prod_i H(k_i)/R
 $$
 and we will show by a d\'evissage argument that all of them are  isomorphisms.

 Let $j$ be a an integer satisfying $1 \leq j \leq n-1$.
 According to a variant of \cite[Lemma 8.3]{GPS}, we have an exact sequence of fppf (resp.\ \'etale, Zariski)
sheaves on $\Spec(k)$
$$
0 \to \WW\Bigl( (t_j)_*\bigl( \Lie(H)(A_{j-1}) \otimes_{A_{j-1}} J_j \bigr) \Bigr)
\to R_{A_j/k}(H_{A_j}) ) \to R_{A_i/k}(H_{A_{j-1}} ) \to 1  .
$$
 where $t_j: \Spec(A_j) \to  \Spec(k)$ is the structural morphism.
 We have  then a sequence of $k$--algebraic groups
 $$
 0 \to (\GG_a)^{m_j} \to G_j \to G_{j-1} \to 1.
 $$
 The map $G_j \to G_{j-1}$ is a $(\GG_a)^{m_j}$-torsor
 which is trivial since $G_{j-1}$ is affine. We have then 
 a decomposition of $k$-schemes $ G_j \simlgr G_{j-1}  \times_k  (\GG_a)^{m_j}$ 
 which induces an isomorphism 
 $G_j(k)/R \to G_{j-1}(k)/R$. Thus the map $G(k)/R \to  G_0(k)/R = \prod_i H(k_i)/R$
 is an isomorphism as desired.

\smallskip

\noindent (2) We have $A/I=k_1 \times \dots \times k_d$.
Lemma \ref{lem_GS} provides an isomorphism $H(A)/H(A)^+ \simlgr H(A/I)/H(A/I)^+
=\prod_i H(k_i)/H(k_i)^+$ and we have an isomorphism $H(k_i)/H(k_i)^+ \simlgr H(k_i)/RH(k_i)$
for each $i$ \cite[thm. 7.2]{Gi1}.
This completes the proof by chasing diagram.
\end{proof}

 \section{Patching, R-equivalence and twisted flag schemes}
 
 \subsection{Using the implicit function theorem}
 We consider  a variation of the framework of \cite[2.4]{HHK}.
 Let $T$ be a complete DVR of fraction field $K$ and $t$ be an unformizing parameter. Let $\widehat R_0$ be ring
 containing $T$ and  which is also a complete discrete valuation ring having uniformizer $t$. We denote by $F_0$ the fraction field of $\widehat R_0$.
 
Let $\alpha>1$ be a real number and we define the absolute value on $F_0$ by 
$\mid\! t^n u \!\mid= \alpha^{-n}$ for $n \in \ZZ$ and $u \in 
(\widehat R_0)^\times$.
 
 Let $F_1$, $F_2$ be subfields of $F_0$ containing $T$.
 We further assume that we are given 
  $t$-adically complete $T$-submodules $V \subset F_1 \cap  \widehat R_0$
  and $W \subset  F_2 \cap  \widehat R_0$ 
satisfying the following conditions:

\begin{equation}\label{cond_I}
 V + W = \widehat R_0;
\end{equation}

\vskip-4mm

\begin{equation}\label{cond_II} 
V \cap t\widehat R_0= t V \quad \hbox{ and} \quad
W \cap t\widehat R_0= t W.
\end{equation}

 \noindent Note that Condition \eqref{cond_II}  is equivalent to

\begin{equation}\label{cond_IIbis} 
 V \cap t^n\widehat R_0= t^n V \quad \hbox{and} \quad 
W \cap t^n\widehat R_0= t^n W \quad \hbox{ for \,  each } \quad  n \geq 1.
\end{equation}

\begin{sremark} {\rm Condition \eqref{cond_II} above is added there compared with  \cite[2.4]{HHK}
but we do not require at this stage that $F_1$ is dense in $F_0$.
 }
\end{sremark}

We equip the submodules $V[\frac{1}{t}]$ of $F_0$ 
of the induced metric (and similarly for $W[\frac{1}{t}]$).

\begin{slemma}\label{lem_banach} (1) $V$ is closed in $\widehat R_0$.

\smallskip

\noindent (2) For $v \in V[\frac{1}{t}] \setminus \{0\}$, we have  $$
\mid v\mid= \mathrm{Inf}\{ \alpha^{n} \mid t^n v \in V \} .
$$

\smallskip

\noindent (3) We have $V= \mathrm{Inf}\{ v \in V[\frac{1}{t}]  \mid 
\enskip \mid v \mid  \,  \geq 0 \} $ and $V$ is a clopen submodule of
$V[\frac{1}{t}]$.

\smallskip

\noindent (4) $V[\frac{1}{t}]$ is closed in $F_0$ and is a Banach $K$-space. 
 
\end{slemma}

\begin{proof}
 (1) Our assumption is that the map $V \to \limproj V/t^{m+1} V$ is an isomorphism. Let $(x_n)$ be a sequence of $V$ 
 which converges in  $\widehat R_0$. For each $m \geq 0$, condition \eqref{cond_IIbis} shows that 
 the map $V/t^{m+1} V \to \widehat R_0/t^{m+1} \widehat R_0$ 
 is injective so that the sequence $(x_n)$ modulo $t^{m+1} V$ is stationary to some $v_m  \in V/t^{m+1} V$.
 The $v_m$'s define a point $v$ of $V$ and the sequence $(x_n)$ converges to $v$.
 
 \smallskip

 \noindent (2) We are given $v = t^m v'\in  V[\frac{1}{t}]$  with $v' \in V \setminus t V$
 and we have \break 
  $\mid\! v \! \mid= \mathrm{Inf}\{  \alpha^{n} \mid t^n v \in \widehat R_0 \}
  = \alpha^{m} \mathrm{Inf}\{  \alpha^{n} \mid t^n v' \in \widehat R_0 \}$.
  Condition (II) implies that  $v' \in    \widehat R_0 \setminus t  \widehat R_0$ so that 
  $\mid\! v \! \mid=0$. We conclude that  $\mid\! v \! \mid= \mathrm{Inf}\{ \alpha^{n} \mid t^n v \in V \}$.

 \smallskip
 
 \noindent (3) It readily follows from the assertion (2).

 \smallskip
 
 \noindent (4) Let $(x_n)$ be a sequence of $V[\frac{1}{t}]$ converging to some $x \in F_0$.
 We want to show that $x$ belongs to $V[\frac{1}{t}]$ so that we can assume that $x\not = 0$
 and that $\mid\! x_n \! \mid = \mid\! x \! \mid = \alpha^m$ for all $n \geq 0$.
 Assertion (2) shows that $t^m x_n$ is a sequence of $V$ and (1) shows that its limit
 $t^m x$ belongs to $V$. Thus $x \in V[\frac{1}{t}]$. We have shown that $V[\frac{1}{t}]$ is closed in $F_0$.
 
 Finally since $F_0$ is a Banach $K$--space so is $V[\frac{1}{t}]$.
\end{proof}

The following statement extends partially \cite[th. 2.5]{HHK} and \cite[prop. 4.1]{HHK2}.

\begin{sproposition}\label{prop_analytic}
Let $a,b,c$ be  positive integers.
Let $\Omega \subset (F_0)^a \times (F_0)^b$ be an open
neighborhood of $(0,0)$ and let 
$f: \Omega \to (F_0)^c$ be an analytic map.
We denote by $f^a: \Omega \cap (F_0)^a \to (F_0)^c$
and $f^b: \Omega \cap (F_0)^b \to (F_0)^c$. We assume that

\smallskip

(i) $f(0,0)=0$

\smallskip

(ii) the differentials $Df^a_{0}: (F_0)^{a} \to (F_0)^c$ and   $Df^b_{0}: (F_0)^{a} \to (F_0)^c$
satisfy $$
Df^a_{0}\Bigl( V[\frac{1}{t}]^a \Bigr) + Df^b_{0}\Bigl( V[\frac{1}{t}]^b\Bigr)=(F_0)^c.
$$

\smallskip

\noindent Then there is a real number
$\epsilon  > 0$ such that for all $y \in (F_0)^c$ with $\mid \! y \! \mid \, 
 \leq \epsilon$, there exist
$v \in  V^a$ and $w \in  W^b$ such that
$(v, w) \in  \Omega$  and $f(v, w) = y$.

\end{sproposition}

\begin{proof}
We consider the continuous embedding $i: V[\frac{1}{t}]^a \times W[\frac{1}{t}]^b  \to (F_0)^a \times (F_0)^b$ and define $\widetilde \Omega=i^{-1}(\Omega)$ and 
the function $\widetilde f= f \circ i : \widetilde \Omega \to (F_0)^c$.

\begin{sclaim}  The map $\widetilde f$ 
is strictly differentiable at $(0,0)$ and
$D\widetilde{f}_{(0,0)}:  V[\frac{1}{t}]^a \times W[\frac{1}{t}]^b \to (F_0)^c$ is onto. 
\end{sclaim}

Since $f$ is $F_0$--analytic at $(0,0)$, it is strictly differentiable  \cite[I.5.6]{Sc}, that is, there exists
an open  neighborhood $\Theta$ of $(0,0)$ and a positive real number $\beta$ such that 
$$
\mid \! f(x_2) -f(x_1) - Df_{(0,0)}.(x_2-x_1) \! \mid \enskip \leq \enskip \beta 
\mid \! x_2 -x_1 \! \mid \qquad \forall x_1,x_2 \in \Theta.
$$
It is then strictly derivable as function between the 
Banach $K$--spaces $(F_0)^a \times (F_0)^b \to (F_0)^c$.
On the other hand the embedding  $i: V[\frac{1}{t}]^a \times W[\frac{1}{t}]^b  \to (F_0)^a \times (F_0)^b$
is $1$--Liftschitz so is strictly differentiable at $(0,0)$. As composite of strictly differentiable  functions,
$\widetilde f$ is strictly differentiable at $(0,0)$ \cite[\S 1.3.1]{BF}. Furthermore
the differential 
$D\widetilde{f}_{(0,0)}$  is the composite of
$$
V[\frac{1}{t}]^a \times W[\frac{1}{t}]^b \xrightarrow{\quad i \quad} (F_0)^a \times (F_0)^b \xrightarrow{\enskip Df^a_{0} + Df^b_{0} \enskip} (F_0)^c.
$$
Condition (ii) says exactly that $D\widetilde{f}_{(0,0)}$  is surjective. The Claim is proven.

We apply the implicit function theorem to the function $\widetilde f$ \cite[\S 1.5.2]{BF}
(see  \cite[\S 4]{Sc} for concocting a proof). Lemma \ref{lem_banach}.(4)
shows that $V[\frac{1}{t}]^a \times W[\frac{1}{t}]^b $ is a Banach $K$--space and so is $F_0$.
There exists then an open neighborhood $\Upsilon \subset \widetilde \Omega$
of $(0,0)$ in $V[\frac{1}{t}]^a \times W[\frac{1}{t}]^b$  such that $\widetilde f_{\mid \Upsilon}$ is open.
Up to shrink $\Upsilon$ we can assume that 
$\Upsilon \subset V \times W$ according to Lemma \ref{lem_banach}.(3).
There exists then a real number $\epsilon >0$ 
such that 
$$
\bigl\{ y \in (F_0)^c \mid \enskip  \mid \! y \! \mid \leq  \epsilon \bigr\} \enskip \subset \enskip
\bigl\{ y \in (F_0)^c \mid \enskip  \mid \! y \! \mid  < 2 \epsilon \bigr\} \subset \widetilde f(\Upsilon).
$$
We conclude that 
for all $y \in (F_0)^c$ with $\mid \! y \! \mid \, 
 \leq \epsilon$, there exist
$v \in  V^a$ and $w \in  W^b$ such that
$(v, w) \in  \Omega$  and $f(v, w) = y$.
\end{proof}

 \newpage

\begin{scorollary}\label{cor_analytic}
Let $n$ be a positive integer.
Let $\Omega \subset (F_0)^n \times (F_0)^n$ be an open neighborhood of $(0,0)$ and let 
$f: \Omega \to (F_0)^n$ be an analytic map which satisfies 

\smallskip

(i) $f(0,0)=0$

\smallskip

(ii) $f(x,0)=f(0,x)=x$ over an open neighborhood $\Upsilon$ of $0$.

\smallskip

\noindent Then there is a real number
$\epsilon  > 0$ such that for all $a $ with $\mid \! a \! \mid \, 
 \leq \epsilon$, there exist
$v \in  V^n$ and $w \in  W^n$ such that
$(v, w) \in  \Omega$  and $f(v, w) = a$.

\end{scorollary}

\begin{proof} In this case we have $a=b=c=n$ and $Df^a_{0} = Df^b_{0}= \mathrm{Id}_{ (F_0)^n}$
so that Proposition \ref{prop_analytic} applies.
\end{proof}

\subsection{Kneser-Tits' subgroups}

Continuing in the previous setting, 
we assume furthermore that $F_1$ is $t$-adically dense
in $F_0$. Let $F \subset F_1 \cap F_2$ be a subfield. 
For dealing later with Weil restriction issues it is 
convenient to deal with  a finite field extension $E$ of $F$.

\begin{sproposition}\label{prop_KT} 
 Let  $H$ be a semisimple simply connected
$E$--group scheme assumed strictly isotropic.
We put $G=R_{E/F}(H)$. 
For each overfield $L$ of $F$, we put 
$G(L)^+= H(L \otimes_F E)^+$ where the second group 
is that defined  in \S \ref{subsec_PS}.
 Then  we have the decomposition
$$
G(F_0)^+= G(F_1)^+ \, G(F_2)^+.
$$
\end{sproposition}

\begin{proof} Without lost of generality we can assume that $F$ is infinite. The proof is based on an analytic argument requiring 
some preparation. We put $d= [E:F]$ and fix
an isomorphism $E \cong F^d$ of $F$--vector spaces.

Let $P$ be a strictly proper parabolic $E$--subgroup of $H$. 
Let $U$ be its 
unipotent radical and $U_{last}$ the last part of Demazure's filtration
\cite[\S 3.2]{GPS}. Let \break $u: \GG_{a,E}^d \simlgr U_{last}$ be an $E$--group isomorphism.
According to \cite[Lemma 3.3.(3)]{GPS}, $H(E)^+. \Lie(U_{last})(E)$
generates $\Lie(H)(E)$ as a vector space over $E$.
There exists $h_1,\dots, h_n \in H(E)^+$ such that 
$$
\Lie(H)(E)= \,  {^{h_1}\!\Lie}(U_{last})(E) \, + \,    ^{h_2}\!\Lie(U_{last})(E) \, + \, \dots 
\, + \,   ^{h_n}\!\Lie(U_{last})(E) . $$
We consider the map $b:  (\GG_{a,E}^d)^n \to H$, $b(x_1,\dots, x_n)= \, {^{h_1}u}(x_1) \,  \dots \,  
^{h_n}\!u(x_n)$. Its differential at $0$ is 
$db_{0}: E^{dn} \cong \Lie(U_{last})(E)^n \to \Lie(H)(E)$, $(X_1,\dots, X_n) \mapsto \sum_{i=1}^n \, \,   ^{h_i}\!X_i$, 
so is surjective. We cut  now 
the affine space  $(\GG_{a,E}^d)^n$
by some suitable affine $E$--subspace $\GG_{a,E}^r$
such that the restriction $b'$ of $b$ to  $(\GG_{a,E}^d)^n$ is such that $db'_{0}$ is an isomorphism.
 The map $b'$ is then \'etale at a neighborhood of $0$.
It follows that $b'_\sharp= R_{E/F}(b'): R_{E/F}\bigl(\GG_{a,E}^r\bigr) \to G=R_{E/F}(H)$
is \'etale also at a neighborhood of $0$ \cite[A.5.2.(4)]{CGP}.

Since the field $F_0$ is henselian,  
the local inversion theorem holds \cite[prop. 2.1.4]{GGMB}. 
We mean that there exists an open neighborhood $\Upsilon \subset (E \otimes_F F_0)^r$ such that
 the restriction 
$b'_{\sharp \mid \Upsilon}: \Upsilon \to G(F_0)$ is a topological open embedding.

We consider now the product morphism $q: \Upsilon \times \Upsilon \to G(F_0)$,
$q(x,y)= b'_\sharp(x)\,  b'_\sharp(y)$. We put  $\Omega= q^{-1}( b'_\sharp(\Upsilon))$, this an open
subset of $(E \otimes_F F_0)^{2r} \cong (F_0)^{2dr}$.
Then the restriction $q_{\mid \Omega}$ defines an (unique) analytical map
$f:  \Omega \to \Upsilon$ 
such that $q(x,y)= b'_\sharp( f(x,y) )$.

By construction we have $f(0,x)= f(x,0)=x$ for $x$ in a neighborhood of $0 \in (E \otimes_F F_0)^r \cong F_0^{dr}$.
We apply Corollary \ref{cor_analytic} to $f$ so that there exists $\epsilon >0$
such that  for each $a \in \Upsilon$ with $\mid a \mid  \,  \leq \,  \epsilon$, there exist
$v \in  W^{dr}$ and $w \in  W^{dr}$ such that
$(v, w) \in  \Omega$  and $f(v, w) = a$.
We denote by $\Upsilon_\epsilon= \Upsilon \cap B(0, \epsilon)$.
Then $b'_\sharp( \Upsilon_\epsilon)$ is an open neighborhood of $0$ in $(F_0)^{r}$.

Let us now prove that $G(F_0)^+= G(F_1)^+ \times G(F_2)^+$.
Since $F_1$ is dense in $F_0$, 
$G(F_1)^+$ is dense in $G(F_0)^+$.
It is then enough to show that $\Upsilon_\epsilon  \subset G(F_1)^+ \times G(F_2)^+$.
Let $g =b'_\sharp(a) \in b'_\sharp( \Upsilon_\epsilon)$.
Then  $a= f(v,w)$ with $(v, w) \in  \Omega$.
It follows that $g =b'_\sharp( a)= b'_\sharp( f(v,w)) = q(v,w) = b'_\sharp(v) \,  
b'_\sharp(w) \in  G(F_1)^+ \times G(F_2)^+$.
\end{proof}

This could be refined as follows.

\begin{sproposition}\label{prop_R_eq} Let 
$H$ be a reductive $E$-group and put $G=R_{E/F}(H)$.

\smallskip

\noindent (1) $RG(F_1) \, RG(F_2)$ contains an open neighborhood of $1$  in $G(F_0)$.

\smallskip

\noindent (2) If $RG(F_1)$ is dense in $RG(F_0)$, then  $RG(F_1) \, RG(F_2)=RG(F_0)$.

\smallskip

\noindent (3) If $H$ is semisimple simply connected and $H_{F_1 \otimes_F E}$ is 
strictly isotropic, then we have
$$
G(F_1)^+ \, RG(F_2)=G(F_0)^+ .
$$

\smallskip

\noindent (4) If $H$ is semisimple simply connected and $H_{F_i \otimes_F E}$ is strictly isotropic for $i=1,2$,
then $G(F_1)^+ \, G(F_2)^+=G(F_0)^+$.

\end{sproposition}

The subgroups $G(F_1)^+$, $G(F_0)^+ $ are defined as in Proposition \ref{prop_KT}.
 
\begin{proof} Once again we put $d= [E:F]$ and fix
an isomorphism $E \cong F^d$ of $F$--vector spaces.

\smallskip

 \noindent (1) Let $T \subset H$ be a maximal $E$--torus and let $1 \to S  \to Q \xrightarrow{s} T \to 1$
 be a resolution of $T$  where $Q$ is a quasitrivial torus and $S$ is a torus.
 We have $Q=R_{C/E}(\GG_m)$ where $C$ is an \'etale $E$--algebra so that
 $Q$ is an open subset of the affine $E$--space $\WW(C)$.

 We use now Raghunathan's technique \cite[\S 1.2]{R}.
 There exists $h_1,\dots, h_n \in H(E)$ such that 
$\Lie(H)(E)=  \,  ^{h_1}\!\Lie(T)(E) \, + \,   ^{h_2}\!\Lie(T)(E) \, + \, \,  \dots \, + \,     ^{h_n}\!\Lie(T)(E) $.
We define the open $E$-subvariety $U \subset
\WW_E(C^n)$ where $1+x_1, \dots , 1+x_n$ belongs to 
$Q$.  
We consider the map $b: U \to H$, $b(x_1,\dots, x_n)= \, ^{h_1}\!s(1+x_1) \,  \dots \,  
^{h_n}\!s(1+x_n)$. Its differential at $0$ is  
$$
db_{0}: C^n \to \Lie(H)(E), \quad (c_1,\dots, c_n) \mapsto
 \sum\limits_{i=1}^n  \, \,  ^{h_i}\!ds(c_i),
 $$ 
and is onto (observe that $\Lie(Q)(E) \to \Lie(T)(E)$ is surjective since $Q \to E$ is smooth). 
We cut  now the $E$-affine space $\WW_E(C^n)$ by some suitable affine $E$--subspace $\WW_E(E^r)$ of  $\WW_E(C^n)$ 
such that the restriction $b'$ of $b$ to 
the $E$--variety $Y= U \cap \WW_E(E^r)$
is such that  $db'_{0}: \mathrm{Tan}_{Y,0} \to \Lie(H)(E)$ is an isomorphism.
Note that $Y$ and $H$ have same dimension $r$ over $E$.

The map $b'$ is then \'etale at a neighborhood of $0$.
It follows that \break $b'_\sharp= R_{E/F}(b'): 
X=R_{E/F}\bigl(Y\bigr) \to G=R_{E/F}(H)$
is \'etale also at a neighborhood of $0$ \cite[A.5.2.(4)]{CGP}.
Note that $X$ is an open subset of $R_{E/F}\bigl(\WW_E(E^r)\bigr) \cong \WW(E^r) \cong \mathbf{A}^{dr}_{F}$.

Since the field $F_0$ is henselian,  
the local inversion theorem holds \cite[prop. 2.1.4]{GGMB}. 
We mean that there exists an open neighborhood $\Upsilon \subset X(F_0) \,  \subset \,
(F_0 \otimes_F E)^r \cong (F_0)^{dr}$
 of $0$ such that the restriction 
$b'_{\sharp \mid \Upsilon}: \Upsilon \to G(F_0)$ is a topological open embedding.

We consider now the product morphism $q: \Upsilon \times \Upsilon \to G(F_0)$,
$q(x,y)= b'_\sharp(x) \, b'_\sharp(y)$. We put  $\Omega= q^{-1}( b'_\sharp(\Upsilon))$, this an open
subset of $(F_0)^{2dr}$.
Then the restriction $q_{\mid \Omega}$ defines a (unique) analytical map $f: 
 \Omega \to \Upsilon$ 
such that $q(x,y)= b'_\sharp( f(x,y) )$.

By construction we have $f(0,x)= f(x,0)$ for $x$ in a neighborhood of \break $0 \in (F_0 \otimes_F E)^r \cong F_0^{dr}$.
We apply Corollary \ref{cor_analytic} to $f$ 
 so that there exists $\epsilon >0$
such that  for each $a \in \Upsilon$ with $\mid a \mid  \,  \leq \epsilon$, there exist
$v \in  V^{dr}$ and $w \in  W^{dr}$ such that
$(v, w) \in  \Omega$  and $f(v, w) = a$.
We denote by $\Upsilon_\epsilon= \Upsilon \cap B(0, \epsilon)$.
Then ${b'_\sharp}( \Upsilon_\epsilon)$ is an open neighborhood of  $1$ in $(F_0)^{dr}$.

We claim that $\Upsilon_\epsilon \subset RG(F_1) \, RG(F_2)$.
Let $g =b'_\sharp(a) \in b( \Upsilon_\epsilon)$.
Then  $a= f(v,w)$ with $(v, w) \in  \Omega$.
It follows that $g= b'_\sharp( a)= b'_\sharp( f(v,w)) = q(v,w) = b'_\sharp(v) \, b'_\sharp(w) \in  RG(F_1) \times RG(F_2)$.

 \smallskip
 
 \noindent (2)  
 If $RG(F_1)$ is furthermore dense in $RG(F_0)$, then  (1) shows that 
 $RG(F_1) \, RG(F_2)$ is a dense open subset of $RG(F_0)$.
 Thus $RG(F_1) \, RG(F_2)=RG(F_0)$

 \smallskip
 
 \noindent (3) We assume that $H$ is semisimple simply connected and 
 that  $G_{F_1 \otimes_F E_1}$ is strictly isotropic.
 According to Lemma \ref{lem_GS}.(2), we have
 $G(F_1)^+=RG(F_1)=RH(F_1 \otimes_F E)=H^+(F_1 \otimes_F E)$
 and similarly for $F_0$.

Since $H^+(F_1 \otimes_F E)^+$ is dense in $H^+(F_0 \otimes_F E)$, it follows that 
$RG(F_1)$ is dense in $RG(F_0)$. Assertion  
  (2) yields then  $RG(F_1) \, RG(F_2)=RG(F_0)$.
  Lemma \ref{lem_R_eq}.(2) states that $G(F_1)^+=RG(F_1)$
  and $G(F_0)^+=RG(F_0)$. We conclude that 
  $G(F_1)^+ \, RG(F_2)=G(F_0)^+$.

 \smallskip
 
 \noindent (4) We have  furthermore  $G(F_2)^+= RG(F_2)$ so that (3)
 yields  $G(F_1)^+ \, G(F_2)^+ =G(F_0)^+$.
\end{proof}

 \begin{sremark} {\rm Note that (1) shows in particular that  
 $RG(F_0)$ is an open subgroup of $G(F_0)$.
  }
 \end{sremark}

 The main result of the section is the following patching statement on twisted flag varieties.

 \begin{stheorem} \label{thm_key} We put $F=F_1 \cap F_2$. 
 Let $H$ be a reductive $E$--group and let 
 $X$ be a twisted flag $E$--variety of $H$.
  We put $G=R_{E/F}(H)$ and $Z=R_{E/F}(X)$.
 If $Z(F_1) \not = \emptyset$
 and  $Z(F_2) \not = \emptyset$, then  $Z(F) \not = \emptyset$.
 \end{stheorem}

 \begin{proof} Without loss of generality, we can assume that $H$ is
 semisimple simply connected and that it is absolutely $E$--simple.
 This implies that $H$ is strictly $F_i \otimes_F E$--isotropic for $i=1,2$.
 Let $x_i \in Z(F_i)= X( F_i \otimes_F E)$ and denote by $P_i=\Stab_{G_{F_i}}(x_i)$
 its stabilizer, this is  parabolic $F_i \otimes_F E$--subgroup of $H$.
Since $F_i \otimes_F E$ is a semilocal algebra, we have
 $x_1=h.x_2$ for some $h \in H(E \otimes_F F_0)$ \cite[XXVI.5.2]{SGA3}.
According to Lemma \ref{lem_GS0}, we have 
$H(E \otimes_F F_0)= H^+(E \otimes_F F_0) \,  P_2(E \otimes_F F_0)$
so we can assume that $h \in H^+(E \otimes_F F_0)$.
According to Proposition \ref{prop_R_eq}.(4), 
we can write $h=h_1 \, h_2$ with $h_i \in H^+(E \otimes_F F_i)$ for $i=1,2$. Since $x_1=h.x_2$, we obtain $h_1^{-1}. x_1= h_2^{-1}. x_2$.
This  defines a point of  $X(E)=Z(F)$.
 \end{proof}

 \section{Relation with the original HHK method}
 
We recall the setting.
Let $T$ be an excellent  complete discrete valuation ring with fraction field $K$, residue field $k$ and
uniformizing parameter $t$. 
Let $F$ be a one-variable function field over $K$
and let $\gX$ be a normal model of $F$, i.e. a normal connected
projective $T$-curve with function field $F$.
 We denote by $Y$ the closed fiber of $\gX$ and fix a separable closure $F_s$ of $F$.
   
 For each point $P \in Y$, let $R_P$ be the local ring of
 $\gX$ at $P$; its completion $\widehat R_P$ is a domain 
 with fraction field denoted by $F_P$.

 For each subset $U$ of $Y$ that is contained
 in an irreducible component of $Y$ and does not meet the other components, 
we define
 $R_U= \bigcap\limits_{P \in U} R_P \subset F$. We denote by $\widehat R_U$
 the $t$--adic completion of $R_U$.
 The rings $R_U$ and $\widehat R_U$ are excellent normal domains 
 and we denote by $F_U$ the fraction field of $\widehat R_U$
 \cite[Remark 3.2.(b)]{HHK3}.
 
 Each height one prime $\cp$ in $\widehat R_P$ that contains
  $t$ defines a branch of $Y$ at $P$ lying 
 on some irreducible component of $Y$. 
 The $t$-adic completion  $\widehat R_{\cp}$ of the local ring
 $R_{\cp}$ of $\widehat R_P$ at $\cp$ is a complete DVR
 whose fraction field is denoted by $F_{\cp}$. 
 The field $F_{\cp}$ contains also $F_U$ if $U$ 
 is an irreducible open subset of $Y$ such that $P \in \ol{U} \setminus U$.
 We have then a diagram of fields
 
\[\xymatrix{
 &   F_{\cp}  & \\ 
 F_P \ar[ru] & & F_U  . \ar[lu]
}\]

 \begin{sexample} \label{sexample1}{\rm We assume that $T=k[[t]]$ and 
 take $X= \PP^1_K$, $\gX= \PP^1_T$,
 $P= \infty_k$, $U=\mathbb{A}^1_k= \Spec(k[x])$.
 The ring $R_U$ contains $k[[t]][x]$ and is 
 its localization with respect to the multiplicative set $S$
 of elements which are units modulo $t$.
 The  $t$-adic completion of $R_U$ is   $\widehat R_U=
 k[x][[t]]$; we have $F_U= Frac( \widehat R_U)=
 k(x)((t))$.
 The local ring of $\gX$ at $P=\infty_k$ is $R_P=k[[t]][x^{-1}]_{(x^{-1},t)}$
 so that its completion is $\widehat R_P= k[[t,x^{-1}]]$;
 in particular $F_P=k((t,x^{-1}))$.
 
 We take  $\cp= t \widehat R_P \subset \widehat R_P$
 and  the $t$-adic completion  $\widehat R_{\cp}$ 
 of the local ring
 $R_{\cp}$ of $\widehat R_P$ at $\cp$ is a complete DVR
 which is $k((x^{-1}))[[t]]$. In particular $F_\gp= k((x^{-1}))((t))$.
 }
 \end{sexample}

 \medskip

 \begin{ssetting} \label{setting_hhk} {\rm
 Let $\cP$ be a non-empty finite set of closed points of $Y$ that contains
 all the closed points at which distinct irreducible components meet.
 Let $\cU$ be the set of connected components of $Y \setminus \cP$ and let $\cB$
 be the set of branches of $Y$ at points of $\cP$.
 This yields a finite inverse system of field $F_P, F_U, F_\gp$ (for $P \in \cP$;
 $U \in \cU$, $\gp \in \cB$) where $F_P, F_U \subset F_\gp$ if $\gp$
 is a branch of $Y$ at $P$ lying in the closure of $U$.
 }
\end{ssetting}

\begin{slemma}  \label{lem_dense0} We assume that  $X= \PP^1_K$, $\gX= \PP^1_T$,
 $P= \infty_k$, $U=\mathbb{A}^1_k= \Spec(k[x])$ and $\gp$ the branch
 of $P$. We put $F_1=F_P$, $F_2=F_U$
 and $F_0=F_\gp$.

\smallskip

\noindent (1) $F_1$ is $t$-dense in $F_0$.

\smallskip

\noindent (2)  We put $V= F_1 \cap \widehat R_\cp$ and   $W= F_2 \cap \widehat R_\cp$.
Then $V$ and $W$ satisfy conditions \eqref{cond_I} and \eqref{cond_II}.
 
\end{slemma}

 \begin{proof}
 \noindent  (1) We are given $u_0/v_0 \in F_0$ with $u_0, v_0 \in \widehat R_\cp$,  $v_0 \not =0$.
 There exists elements $u, v \in R_\gp$ very close respectively of $u_0,v_0$ with $v \not =0$.
 Let $s_1, s_2 \in R_p \setminus R_p \gp$ such that $s_1 u \in R_P$ and $s_2 v \in R_\gp$.
 Then $u_0/v_0$ is very close of $(s_1s_2 u) / (s_1s_2 v) \in F_1$.

  \sm
  
  \noindent (2) Condition \eqref{cond_II} is obviously fullfilled since $t \in F_1 \cap F_2$.
  For establishing condition \eqref{cond_I}, we are given an element $f$ of $\widehat R_\gp$ and 
  may write it as $f= \sum\limits_{i=0}^\infty x^{m_i}  \Bigl( \sum_{j=0}^\infty a_{i,j} \frac{1}{x^j} \Bigr) t^i$
  where the $m_i$'s are non-negative integers and $a_{i,j} \in T$. 
  We decompose $$
  f= f_1 + f_2=  \sum_{i=0}^\infty x^{m_i} \Bigl( \sum_{j=m_i}^\infty a_{i,j} \frac{1}{x^j} \Bigr) t^i
  \quad +  \quad
   \sum_{i=0}^\infty x^{m_i}  \Bigl(\sum_{j=0}^{m_i-1} a_{i,j} \frac{1}{x^j} \Bigr)  t^i 
  $$
  We observe that $f_2$ belongs to $\widehat R_P$ so belongs to $V$.
We recall that $R_U$ is the localization of $T[x]$ 
with respect to the elements which are units modulo $t$. 
We conclude that $f_2$ belongs to $W$ as desired.
 \end{proof}

 \begin{stheorem} \label{thm_main_hhk} Let $G$ be a reductive $F$--algebraic group.

 \smallskip
 
 \noindent (1) Let $Z$ be a twisted flag projective $F$--variety for $G$.
 Then $Z(F) \not= \emptyset$ if and only if $Z(F_U) \not = \emptyset$ for 
 each $U \in \cU$ and  $Z(F_P) \not = \emptyset$ for 
 each $P \in \cP$.
 
 \smallskip
 
 \noindent (2) For each $U \in \cU$ (resp.\ each $P \in \cP$), 
 we fix an $F$--embeddings $i_U: F_s \to F_{U,s}$  (resp.\ $i_P: F_s \to F_{P,s}$)
 providing identifications $\Delta(G_{F_s}) \simlgr \Delta(G_{F_{U,s}})$
 (resp.\, $\Delta(G_{F_s}) \simlgr \Delta(G_{F_{P,s}})$).
 The Tits index $\Delta_0(G)$ is the smallest subset
  of $\Delta(G_{F_s})$
 which is stable under the $\star$--action of $\Gal(F_s/F)$ and such that 
  $\Delta_0(G) \subset \Delta_0(G_{F_U})$ for each $U \in \cU$ and  $\Delta_0(G) \subset \Delta_0(G_{F_P})$ for  each $P \in \cP$.
 
 \end{stheorem}

 The recollection for star action and Tits index is done in the beginning of the paper.
 
 \begin{proof}
 (1)  We use a Weil restriction argument as in the proof of \cite[Thm. 4.2]{HHK2}.
This involves a finite morphism $f: \gX \to \PP^1_T$ such that 
$\cP=f^{-1}(\infty_k)$. Write $\uF$ for the function field of  $\PP^1_T$, and let $d=[F:F']$. We put
$\uU= \PP^1_k \setminus \{\infty \}$, $\uP= \infty_k$ and
$\underline{\gp}= (U, \cp)$ and
  $F_0=\uF_{\underline{\gp}}$, $F_1=\uF_{\uP}$ and $F_2=F_{\uU}$.  Also patching holds for the diamond $(\uF, F_1,F_2,F_0)$ according to \cite[thm 3.9]{HH}
  so in particular $K(x)= \uF= F_1 \cap F_2 \subset F_0$. 
  We put $V= F_1 \cap \widehat R_\cp$ and   $W= F_2 \cap \widehat R_\cp$.
  Lemma \ref{lem_dense0} shows that 
  $F_1$ is dense in $F_0$ and that  $V$ and $W$ satisfy conditions \eqref{cond_I} and \eqref{cond_II}.

  We consider the Weil restriction $\uG= R_{F/\uF}(G)$, it is a reductive $\uF$--group
which acts on the $\uF$--variety $\uZ=R_{F/\uF}(Z)$.
We have $$
\uZ(F_1)= Z(F_1 \otimes_{\uF} F)  = \prod_{P \in \cP} Z(F_P)
$$
according to \cite[Lemma 6.2.(a)]{HH}.
Similarly we have
$$
\uZ(F_2)= Z(F_2 \otimes_{\uF} F)  = \prod_{U \in \cU} Z(F_U)
$$
Our assumptions imply that $\uZ(F_1) \not = \emptyset$ and 
$\uZ(F_2) \not = \emptyset$. Theorem \ref{thm_key} implies 
that $\uZ(\uF) \not = \emptyset$. Thus $\uZ(\uF)= Z(F)$ is non-empty.

  \smallskip  
  
  \noindent (2) Let $\Theta$ be the smallest subset of $\Delta(G_{F_s})$
 which is stable under the $\star$--action of $\Gal(F_s/F)$ and such that 
  $\Delta_0(G) \subset \Delta_0(G_{F_U})$ for each $U \in \cU$ and  $\Delta_0(G) \subset \Delta_0(G_{F_P})$ for 
  each $P \in \cP$.
  We observe that $\Delta_0(G) \subset \Theta$ since it is 
   stable under the star action and 
  satisfies $\Delta_0(G) \subset \Delta_0$ for each $U \in \cU$ and  $\Delta_0(G) \subset \Delta_0(G_{F_P})$ for  each $P \in \cP$.
  For the converse inclusion we consider the $F$--variety $Z$ of parabolic
  subgroups of type $\Theta$.
  For each $U \in \cU$, we have  $\Theta \subset \Delta_0(G_{F_U})$ so that
  $Z(F_U) \not = \emptyset$; similarly we have 
   $\Theta \subset \Delta_0(G_{F_P})$ for each $P \in \cP$
   so that $Z(F_P) \not = \emptyset$.
  Part (1) yields that $Z(F) \not = \emptyset$. Thus $\Theta \subset \Delta_0(G)$
  and $\Theta = \Delta_0(G)$
 \end{proof}

 \begin{scorollary} \label{cor_main_hhk} Let $G$ be a reductive $F$--algebraic group.
 
 \smallskip
 
 \noindent (1) Let $Z$ be a twisted flag projective $F$--variety for $G$.
 Then $Z(F) \not= \emptyset$ if and only if  $Z(F_P) \not = \emptyset$ for 
 each $P \in Y$.
 
 \smallskip
 
 \noindent (2) For each $P \in Y$, 
 we fix an $F$--embedding  $i_P: F_s \to F_{P,s}$
 providing identifications 
  $\Delta(G_{F_s}) \simlgr \Delta(G_{F_{P,s}})$.
 The Tits index $\Delta_0(G)$ is the smallest subset of $\Delta(G_{F_s})$
 which is stable under the $\star$--action of $\Gal(F_s/F)$ and such that 
$\Delta_0(G) \subset \Delta_0(G_{F_P})$ for  each $P \in Y$.
 
 \end{scorollary}

 \begin{proof}
 (1)  We assume $Z(F_P) \not = \emptyset$ for 
 each $P \in Y$. Let $Y_1,\dots, Y_d$ be the irreducible components of $Y$
 with respective generic points $\eta_1, \dots,  \, \eta_d$. 
 According to  \cite[prop. 5.8]{HHK2}, there exists non-empty affine subsets 
 $U_i \subset Y_i$ ($i=1,\dots, d$) such that $Z(F_{U_i}) \not = \emptyset$
 for $i=1,\dots, d$ and $U_i \cap U_j = \emptyset$ for $i<j$.
 We apply Theorem \ref{thm_main_hhk} to  $\cU= \{ U_1,\dots, U_d\}$, 
 $\cP= Y \setminus \cap_i U_i$ and get that $Z(F) \not = \emptyset$.
 
 \smallskip
 
 \noindent (2) This readily follows of (1).
  
 \end{proof}

 \begin{scorollary} \label{cor_main_hhk2} Let $G$ be a reductive $F$--algebraic group
 and assume that $G$ is the generic fiber of a reductive $\gX$-group scheme $\gG$.
 
 \smallskip
 
 \noindent (1) Let $Z$ be a twisted flag projective $F$--variety for $G$.
 Then $Z(F) \not= \emptyset$ if and only if  $Z(F_v) \not = \emptyset$ for 
 each discrete valuation $v$ of $F$.
 
 \smallskip
 
 \noindent (2) For each discrete valuation $v$ of $F$, 
 we fix an $F$--embedding  $i_v: F_s \to F_{v,s}$
 providing identifications 
  $\Delta(G_{F_s}) \simlgr \Delta(G_{F_{v,s}})$.
 The Tits index $\Delta_0(G)$ is the smallest subset of $\Delta(G_{F_s})$
 which is stable under the $\star$--action of $\Gal(F_s/F)$ and such that 
$\Delta_0(G) \subset \Delta_0(G_{F_v})$ for  each discrete valuation $v$ on $F$.
 \end{scorollary}

 \begin{proof}
 (1) We assume that $Z(F_v) \not = \emptyset$ for 
 each discrete valuation $v$ of $F$.
 Let $\gG$ be a reductive $\gX$--group of generic fiber $G$.
 Without loss of generality we can assume that 
$\gG$ is adjoint. Let $\gG_0$ be the Chevalley form of $\gG$ and  let $(\gB_0,\gT_0)$ be a 
   Killing couple for  $\gG_0$ and let $\Delta_0$ be
   the associated Dynkin diagram.
   Since $\gG_0$ is adjoint, we have 
   an exact sequence of $\ZZ$--group schemes
   \cite[XXIV.1.3 and 3.6]{SGA3}
   $$
   1 \to \gG_0 \to \Aut(\gG_0) \to \Aut(\Delta_0) \to 1 
   $$
The $\Aut(\gG_0)$--torsor
   $\gQ=\mathrm{Isom}(\gG_0,\gG)$ over $\gX$ defines then 
   an $\Aut(\Delta_0)$-torsor over $\gX$.
Since    $\Aut(\Delta_0)$ is a finite constant group,
Grothendieck's theory of fundamental groups tells us that 
   this torsor is the data of 
   a morphism $v: \Pi_1(\gX,\bullet) \to \Aut(\Delta_0)$
   where the base point is $\Spec(F_s) \to \Spec(F) \to \gX$.  
The star action associated to $G$ is the composite
$$
\Gal(F_s/F) \to \Pi_1(\gX,\bullet) \to \Aut(\Delta_0).
$$

 \begin{sclaim}
 $Z$ is the generic fiber of an $\gX$--scheme of parabolic subgroups $\gZ$ of $\gG$. 
 \end{sclaim}
 
 The variety $Z$ is a form of the variety $\Par_I(G_0)$
  of parabolic subgroups of type $I$ where $I \subset \Delta_0$ is stable the star action.
Since $\gX$ is normal, the map  $\Gal(F_s/F) \to \Pi_1(\gX,\bullet)$ is onto so that $I$ is stable under the action of 
$\Pi_1(\gX,\bullet)$ on $\Delta_0$.
  In particular $Q$ admits a reduction $\gQ_I$ to the stabilizer 
   $\Aut_I(\gG_0)$ for the action  of $\Aut(\gG_0)$ on $\Delta_0$. The $\gX$--scheme  $\gZ=\, ^{\gQ_I} \! \! \Par_I(\gG_0)$
    is the scheme of parabolic subgroups of type $I$
    of $\gG= ^{\gQ_I} \! \! \gG_0$, so that $Z$
    is the generic fiber of $\gZ$.

 For applying Corollary \ref{cor_main_hhk}, we have to check
 that $Z(F_P) \not = \emptyset$ for each $P \in Y$.
 If $Q$ is a point of codimension $1$ of $\gX$, 
 it defines a discrete valuation $v_Q$ on $F$ whose completion is  $F_Q$.
 Our assumption implies then that  $Z(F_Q) \not = \emptyset$ in that case.
 We deal now with the case of a closed point $P$ of $\gX$.
 Let $D$ be an irreducible component of $Y=\gX_k$ containing $P$
 and let $Q$ be the generic point of $D$. Since 
 $\gZ$ is proper over $T$, we have $\gZ(\widehat R_Q)=\gZ(F_Q)$ which
 is not empty by the preceding case. It follows that $\gZ_k(k(D)) \not = \emptyset$.
 Again $\gZ_k$ is projective so that $\gZ_k(D) =\gZ_k(k(D))$ is not empty
 and in particular $\gZ_k(k(P))$ is not empty.
 Since $\gZ$ is smooth over $\gX$, the Hensel lemma
 shows that $\gZ_k(\widehat R_{\gX,P}) \to  \gZ_k(k(P))$ is surjective.
 Thus $\gZ(\widehat R_{\gX,P})$ is not empty and so is $Z(F_P)$.

 \smallskip
 
 \noindent (2) It readily follows of (1). 
  
 \end{proof}

 \begin{scorollary} \label{cor_main_hhk3} Let $G$ be a reductive $F$--algebraic group.
 We denote by $\Omega^1_F$ the set of rank one valuations of $F$.
 
 \smallskip
 
 \noindent (1) Let $Z$ be a twisted flag projective $F$--variety for $G$.
 Then $Z(F) \not= \emptyset$ if and only if  $Z(F_v) \not = \emptyset$ for 
 each $v \in \Omega^1_F$.
 
 \smallskip
 
 \noindent (2) For each $v \in \Omega^1_F$, 
 we fix an $F$--embedding  $i_v: F_s \to F_{v,s}$
 providing identifications 
  $\Delta(G_{F_s}) \simlgr \Delta(G_{F_{v,s}})$.
 The Tits index $\Delta_0(G)$ is the smallest subset of $\Delta(G_{F_s})$
 which is stable under the $\star$--action of $\Gal(F_s/F)$ and such that 
$\Delta_0(G) \subset \Delta_0(G_{F_v})$ for  each $v \in \Omega^1_F$.
 
 \end{scorollary}

 \begin{proof}
 (1) We assume that $Z(F_v) \not = \emptyset$ for 
 each $v \in \Omega^1_F$.
  According to \cite[Thm 2.5]{HHKP}, there exists a regular proper model $\gX'$ of $F$
  such that $Z(F_P) \not = \emptyset$ for 
 each $P \in \gX' \times_T k$.  Then Corollary \ref{cor_main_hhk2}.(1)
 shows that $Z(F) \not = \emptyset$.
 
 \smallskip
 
 \noindent (2) This readily follows of (1).
 \end{proof}

\section{Local-global principle for discrete valuations}

Let $T$ be an excellent DVR of fraction field $K$ and residue field $k$.
Let $X$ be a smooth, projective,
geometrically integral curve over $K$. Let $F=K(X)$ be the function field of $X$.

\begin{slemma} \label{lem_extend} Let $\gX$ be a projective, flat curve over $A$
which is connected and regular such that $\gX_K=X$.
Let $\gH$ be a flat  affine $\gX$--group scheme of finite presentation
and assume that  there  exists a  Zariski cover $(\gU_i)_{i \in I}$ of $\gX$ 
 such that each $\gH_{U_i}$ admits a closed embedding $\gH_{U_i} \subset \GL_{n_i}$ such that 
 $\GL_{n_i}/\gH_{U_i}$ is representable by an affine $U_i$--scheme,

Let $\gamma \in H^1_{\fppf}(F,\gH)$ and let 
 $D$ be a divisor of $\gX$ which contains the irreducible
components of $Y= \gX_k$ and such that $\gamma$ extends \textit
to $X \setminus D_K$.

Then there exists a proper birational morphism $q: \gX' \simlgr \gX$  
such that $\gX'$ is a regular proper model of $X$ and 
such that $D'=q^*D$ is  a strict normal crossing divisor  and  
$$
\gamma \in \mathrm{Im}\Bigl( H^1(\gX' \setminus D', \gH) \to H^1(F, \gH) \Bigr).
$$
\end{slemma}

\begin{proof} Using a passage to the limit argument \cite{Mg}, there exists an open affine subscheme 
$\gU_1 \subset \gX \setminus D$
such that $\gamma$ extends to a class $\gamma_1 \in H^1(\gU_1, \gH)$.
According to \cite[cor 1.8]{GP1}, $\gamma_1$ extends to a class
$\gamma_2 \in H^1(\gU_2, \gH)$ where $\gU_2$ is an open subscheme of $\gX  \setminus D$
containing $\gU_1$ and $X \setminus D_K$. 
By purity (i.e. Theorem \ref{thm_extend} of the appendix \ref{app_extend}),
we have $H^1(\gX \setminus D, \gH) = H^1( \gU_2, \gH)$.
Thus $\gamma$ extends over  $\gX \setminus D$.
According to Lipman's theorem we can resolve the singularities of $\gX$
and transform $D$ in a  strict normal crossing divisor, see \cite[lemma 4.7]{HHK}.
\end{proof}

\begin{sproposition} \label{prop_patching}
 Let $G$ be a reductive $F$--group  and assume that
 $p$ does not divide the order of the automorphism group of
 the  absolute root system of $G_{ad}$.
 Let $Z$ be a twisted flag  variety of $G$.
We assume that $Z(F_v) \not = \emptyset$ for 
all discrete valuations of $F$ arising from models of $X$.
Then there exists a regular proper model $\gX$ of $X$ 
with special fiber $Y=\gX_k$ such that for every point
$y \in Y$, then $Z(F_y) \not = \emptyset$.
 \end{sproposition}

 \begin{proof}
 Without loss of generality we can assume that $G$ is adjoint.
   Let $G_0$ be the Chevalley form of $G$ and 
   let $(B_0,T_0)$ be a 
   Killing couple for  $G_0$ and let $\Delta_0$ be
   the associated Dynkin diagram. 
 The variety $Z$ is a form of the variety $\Par_I(G_0)$
  of parabolic subgroups of type $I$ where $I \subset \Delta_0$ is stable
   under the star action defined by the $\Aut(G_0)$--torsor
   $Q=\mathrm{Isom}(G_0,G)$.  
   In particular $Q$ admits a reduction $Q_I$ to the stabilizer 
   $\Aut_I(G_0)$ for the action  $\Aut(G_0)$ on $\Delta_0$ through
the morphism $\Aut(G_0) \to \Out(G_0)\simlgr \Aut(\Delta_0)$.
Furthermore $Z$ is isomorphic to  $^{Q_I} \! \Par_I(G_0)$.
   
We apply  now Theorem 1.1.(b) of \cite{CGR}
to the $\ZZ$--group scheme $\Aut_I(G_0)$. 
It provides a finite $\ZZ$--subgroup $S_0$ of $\Aut_I(G_0)$ such that the map $H^1(F, S_0) \to H^1(F, \Aut(G_0))$ is onto.
Furthermore the construction of $S_0$ is explicit in the proof,
it is an extension of the finite constant group 
 $\Aut_I(G_0)/T_0$ by a finite subgroup of $T_0$.
 In particular $S_0$ is finite  free over $\ZZ$  and our assumption on the characteristic implies that 
$S_{0,T}$ is finite \'etale of rank prime to $p$. 
   It follows that $Q_I$ admits a reduction to an $F$--torsor 
   $E$ under $S_0$.

   The  finite flat $\ZZ$--group scheme $S_{0,A}$ admits 
   a faithful representation $S_0 \hookrightarrow \GL_{N,\ZZ}$
   \cite[\S 1.4.5]{BT} and the quotient 
   $\GL_{N,\ZZ} /S_0$ is representable by an affine $\ZZ$-scheme \cite[\S III.2.6]{DG}.

  According to Lemma \ref{lem_extend},
  there exists  a regular proper model $\gX$ of $X$ and 
  a strict normal crossing divisor  $D$ containing  the irreducible components of $Y$ such that $E$ extends 
  to a $\gX \setminus D$--torsor $\gE$  under $S_0$.
  We put $\gQ_I= \gE \wedge^{S_0}\Aut_I(G_0)$ and consider
  the $\gX \setminus D$-group scheme $\gG= {^{\gQ_I}\!G_0}$ of generic fiber $G$.
  
  We are given a closed point $P \in Y$. If $y$ is of codimension one, by hypothesis, $Z(F_y)$ is not empty. We therefore look at a closed point $P \in Y$.
 We are given a point $P \in Y$ and pick a height one prime $\cp$ in $\widehat R_P$ that contains
  $t$. It  defines a branch of $Y$ at $P$ lying 
 on some irreducible component $Y_1$ of $Y$.

We consider the local ring $A=R_P$ of $\gX$ at $P$ and denote by $A_D$
its localization at $D$. Since $S_{0,T}$ is finite \'etale of degree prime
to $p$, $H^1(A_D,S_0)$ consists in loop torsors as defined in \cite[\S 2.3, lemma 2.3.(2)]{Gi2},  i.e. those arising from cocycles related to tame Galois covers of $A_D$. It follows that the $A_D$--torsor
$Q_I$ is a loop $\Aut_I(G)$--torsor \cite[lemma 2.3.(3)]{Gi2},
so that $\gG \times_{\gX \setminus D} A_D$ is 
by definition a loop reductive group scheme.

Let $F_{P,v}$ be the completion of the field $F_P$ for the valuation associated to the blow-up of $\Spec(A)$ at its closed point.
Our assumption states in particular that $Z(F_{P,v}) \not = \emptyset$,
that is, $G_{F_{P,v}}$ admits a parabolic subgroup of type $I$.
According to \cite[th. 4.1 , (iii) $\Longrightarrow$ (i)]{Gi2}, $\gG \times_{\gX \setminus D} A_D$
admits a parabolic subgroup of type $I$. 
A fortiori $G_{F_P}$ admits a parabolic subgroup of type 
$I$ so that $Z(F_P) \not = \emptyset$.
 \end{proof}

\begin{sremarks}{\rm 
(a) If $p=0$, the result used \cite[th. 4.1]{Gi2}
admits a simple proof, see \cite[Ex. 4.2]{Gi2},
by using the analogous result over Laurent polynomials
\cite[th. 7.1]{GP2}.

\smallskip

\noindent (b) In nice cases inspection of the proof permits 
to weaken the assumption on $p$. 
The precise condition is that the $\Aut(G_0)$--torsor
$\mathrm{Isom}(G_0,G)$ admits a reduction to a finite
$F$--subgroup whose degree is prime to $p$.
For example in type $G_2$, we need to assume only 
that $p$ is prime to $2$.
}
\end{sremarks}

 Together with Corollary \ref{cor_main_hhk}, we obtain the following consequence:

 \begin{stheorem} \label{thm_patching} Let $G$ be a reductive
 $F$--group and assume that $p$ does not divide the order of the automorphism group of the  absolute root system of $G_{ad}$.
 Let $Z$ be a twisted flag $F$--variety of $G$.
 Then $Z(F) \not = \emptyset$ if and only if $Z(F_v) \not = \emptyset$
 for all discrete valuations of $F$ arising from models of $X$.
 \end{stheorem}

 \section{Appendix: characterization of parabolic subgroups.} \label{app_parabolic}
 Let $G$ be a reductive $F$--group 
over a field $F$. We remind the reader 
that an algebraic $F$--subgroup $P$ of $G$ is parabolic
if $P$ is smooth and  $G/P$  is a projective $F$--variety.
The interest of the  probably known statement below
is only in positive characteristic since in this case
there  exist $F$--subgroups $Q$ which are not smooth such that 
 $G/Q$ is a projective $F$--variety \cite{W}. 
  
\begin{sproposition}\label{prop_parabolic} 
Let $P$ be a $k$--subgroup of $G$
such that the quotient variety $G/P$ is projective.
Then the following assertions are equivalent: 

\sm

(i) $P$ is an $F$--parabolic subgroup

\sm

(ii) For each $F$--field $E$, $G(E)$ acts transitively
on $(G/P)(E)$

\sm

(iii) The quotient map $G \to G/P$ admits a rational section
 
\sm

(iv) $P$ is smooth connected

\sm

(v) $P$ is smooth.

\end{sproposition} 
 
 \begin{proof}
\noindent  $(i) \Longrightarrow (ii)$.
Since $(G/P)(E)$ parameterizes the $E$--parabolic
subgroups of $G_E$ of same type that $P$,
 Borel-Tits' conjugacy theorem \cite[th. 4.13.c]{BoT65}
 shows that $G(E)$ acts transitively on $(G/P)(E)$.

\sm

\noindent  $(ii) \Longrightarrow (iii)$.
Our assumption rephrases by saying that
the map $G(E) \to (G/P)(E)$ is onto for each $F$--field $E$.
Applying that to the function field $E=F(G/P)$
of the smooth connected $F$- variety $G/P$
provides a rational section of the map $G \to G/P$.

\sm

\noindent  $(iii) \Longrightarrow (iv)$.
To show the smoothness of $P$ we can assume
that $F$ is algebraically closed.
Then the neutral component  $Q=(P_{red})^0$
of the reduced $F$--subgroup $P_{red}$ of $P$ is smooth. Furthermore 
the quotient $F$--variety $P/Q$ is finite. It follows that the
morphism $q: G/Q \to G/P$ is finite and  a fortiori projective
\cite[Tag 0B3I]{St}. Since the composition of projective morphisms (of qcqs schemes) is projective  \cite[Tag 0C4P]{St}, it follows that 
$G/Q$ is projective. The $F$--subgroup $Q$ of $G$
is then parabolic. Our assumption is that the morphism
$G \to G/P$ has a rational section
 and so 
has a fortiori the finite morphism $q: G/Q \to G/P$.
According to \cite[cor. 6.1.15]{EGA2}, $q$ is  an isomorphism. Thus $Q=P$ and we conclude that $P$ is smooth connected.

\sm

\noindent  $(iv) \Longrightarrow (v)$. Obvious.

\sm

\noindent $(v) \Longrightarrow (i)$.
This is by definition.
\end{proof}

 A variant is the following.

\begin{sproposition}\label{prop_homog} 
Let $X$ be a smooth  projective $G$--variety.
Then the following statements are equivalent: 

\sm

(i) $X$ is a $F$--variety of parabolic subgroups of $G$

\sm

(ii) For each $F$--field $E$, $G(E)$ acts transitively
on $X(E)$.

\end{sproposition}

\begin{proof}
The implication $(i) \Longrightarrow (ii)$ is again
Borel-Tits' conjugacy theorem.
We assume (ii). According to \cite[prop. 1.3]{MPW},
we can assume that $F$ is separably closed.
Since $X$ is smooth, we have $X(F) \not = \emptyset$ 
and denote by  $P$ the stabilizer of some $F$--point $x$.
According to \cite[prop. III.3.2.1]{DG}.
condition (ii) implies that the orbit map $G \to X$, $g \mapsto g.x$
induces an isomorphism $f_x: G/P \simlgr X$.
Proposition \ref{prop_parabolic},  
$(ii) \Longrightarrow (i)$, shows that 
$P$ is a $F$--parabolic subgroup.
Thus $X$ is a $F$--variety of parabolic subgroups of $G$.
\end{proof}

 \begin{sremark} {\rm
 The condition (ii) is called
{\it transitive action of $G$ on $X$}
by Harbater-Hartmann-Krashen.
It   occurs in \cite[th. 3.7]{HHK}. 
Projective homogeneous varieties in the result quoted above are exactly
the various varieties of parabolic subgroups.
 }
 \end{sremark}

 \section{Appendix: extending torsors} \label{app_extend}
 
 We come back to a  purity result of Colliot-Th\'el\`ene and Sansuc.
 
\newpage

 \begin{stheorem}\label{thm_extend}
 Let $X$ be a regular scheme of dimension $2$.
 Let $U$ be an open subcheme of $X$ which contains  $X^{(1)}$. 
 Let $G$ be an affine $X$--group scheme. In the following  cases
 
 \smallskip
 
 (i) $G$ is reductive,
 
 \smallskip
 
 (ii) There exists a  Zariski cover $(U_i)_{i \in I}$ of $X$ 
 such that each $G_{U_i}$ admits a closed embedding $G_{U_i} \subset \GL_{n_i}$ such that 
 $\GL_{n_i}/G_{U_i}$ is representable by an affine $U_i$--scheme,
 
  \smallskip
  
then we have the equality $H^1_{\fppf}(X,G) \simlgr H^1_{\fppf}(U,G)$.

 \end{stheorem}

 \begin{proof}
  The case (i) is \cite[th. 6.13]{CTS}. The case (ii) goes by inspection of 
  the proof. 
 \end{proof}

\bigskip

\medskip


\bigskip
 
 
\begin{thebibliography}{99}







\bibitem[B-T]{BoT65} A.\,Borel, J.\,Tits, {\it Groupes r\'eductifs}, 
Publ.\,Math.\,IHES {\bf 27}(1965),   55--151.



\bibitem[B]{BF} N. Bourbaki, {\it  El\'ements de math\'ematique : fascicule de r\'esultats, Vari\'et\'es diff\'erentielles et analytiques }, Springer.
    




\bibitem[Br-T]{BT} F.  Bruhat, J. Tits, {\it 
Groupes r\'eductifs sur un corps local : II. Sch\'emas en groupes. 
Existence d'une donn\'ee radicielle valu\'ee},
Publications Math\'ematiques de l'IH\'ES {\bf 60} (1984), 5-184.

\bibitem[C-G-R]{CGR} V. Chernousov, P. Gille,   Z. Reichstein, {\it  Reduction of structure for torsors over semi-local rings}, Manuscripta Mathematica {\bf 126} (2008), 465-480.




\bibitem[C-T-O-H-H-K-P-S]{CTOHHKPS} J.-L. Colliot-Th\'el\`ene, 
D. Harbater, J. Hartmann, D. Krashen, R. Parimala, V. Suresh,
{\it Local-global principles for tori over arithmetic curves},
preprint (2019), 
\href{https://arxiv.org/abs/1906.10672}{arXiv:1906.10672}.

 \bibitem[C-T-P-S]{CTPS} J.-L. Colliot-Th\'el\`ene,  R. Parimala, V. Suresh, 
 {\it Patching and local-global principles for
 homogeneous spaces over function fields of $p$-adic curves},
 Comment. Math. Helv. {\bf 87} (2012), 1011-1033. 

 
\bibitem[C-T-S]{CTS} J.-L. Colliot-Th\'el\`ene, J.-J.  Sansuc,
{\it  Fibr\'es quadratiques et composantes connexes r\'eelles},
Mathematische Annalen {\bf 244} (1979), 105-134.


\bibitem[C-G-P]{CGP} B. Conrad, O. Gabber, G. Prasad, {\it
Pseudo-reductive groups},   Cambridge University Press, second edition (2016).


\bibitem[D-G]{DG} M.  Demazure, P. Gabriel, {\it Groupes alg\'ebriques},   North-Holland (1970).
 

\bibitem[EGAI]{EGA1} A.\ Grothendieck, J.-A.\ Dieudonn\'e, {\it El\'ements 
de g\'eom\'etrie alg\'ebrique. I}, Grundlehren der Mathematischen Wissenschaften  166; Springer-Verlag, Berlin, 1971.


\bibitem[EGAII]{EGA2} A.\ Grothendieck (avec la collaboration de J.\ Dieudonn\'e), {\it El\'ements de G\'eom\'etrie Alg\'ebrique II}, Publications math\'ematiques de l'I.H.\'E.S.  no 8 (1961).








\bibitem[Gi1]{Gi1} P. Gille, {\it Le probl\`eme de Kneser-Tits}, expos\'e Bourbaki n0 983,  Ast\'erisque {\bf 326}  (2009), 39-81.


\bibitem[Gi2]{Gi2} P. Gille, {\it 
Loop group schemes and Abhyankar lemma}, preprint (2023),
hal-03933634.



 \bibitem[G-G-MB]{GGMB} P. Gille, O. Gabber,  L. Moret-Bailly,
 {\it Fibr\'es principaux sur les corps hens\'eliens}, Algebraic Geometry {\bf 5} (2014), 573-612.


\bibitem[G-P-S]{GPS} P. Gille,  R. Parimala, V. Suresh, 
 {\it  Local triviality for G-torsors}, Mathematische Annalen {\bf 380} (2021), 539-567.
 
 
\bibitem[G-P1]{GP1} P. Gille, A. Pianzola, {\it
Isotriviality and \'etale cohomology of 
Laurent polynomial rings},  Journal of Pure and Applied Algebra {\bf 212} (2008), 780-800.
 
 \bibitem[G-P2]{GP2} P. Gille, A. Pianzola, {\it
  Torsors, reductive group schemes and extended affine Lie algebras},  Memoirs of  AMS  1063 (2013),
  
\bibitem[G-S]{GS} P. Gille, A. Stavrova, {\it  R-equivalence on reductive group
schemes}, preprint (2021). 


 
 
  \bibitem[H-H]{HH} D. Harbater, J. Hartmann, {\it  Patching over fields},
  Israel J. Math. {\bf 176} (2010), 61–107. 
 
 \bibitem[H-H-K]{HHK} D. Harbater, J. Hartmann, D.
  Krashen, {\it  Applications of patching to quadratic
  forms and central simple algebras}, 
  Invent. Math. {\bf 178} (2009), 231-263.
  
  
 
 \bibitem[H-H-K2]{HHK2} D. Harbater, J. Hartmann, D.
  Krashen, {\it 
  Local-global principles for torsors over arithmetic curves},
   Amer. J. Math. {\bf  137} (2015), 1559-1612.
  
 
 \bibitem[H-H-K3]{HHK3} D. Harbater, J. Hartmann, D.
  Krashen, {\it Refinements to patching and applications to field invariants},  Int. Math. Res. Not. IMRN {\bf  20} (2015), 10399-10450. 
  
  \bibitem[H-H-K-P]{HHKP} D. Harbater, J. Hartmann,
  V. Karemaker, F. Pop, {\it   A comparison between obstructions to local-global principles over semi-global fields},
  Abelian varieties and number theory, 135–146, Contemp. Math. {\bf 767} (2020), Amer. Math. Soc.,  RI. 
  
  
 
 

 
    
    

 
 \bibitem[Mg]{Mg} B. Margaux, {\it  Passage to the limit in non-abelian \v{C}ech cohomology},
 J. Lie Theory {\bf 17} (2007), 591-596.
 
 \bibitem[M-P-W]{MPW} A. S. Merkurjev,  I. A.  Panin,
 A. Wadsworth, {\it  Index reduction formulas for twisted flag varieties. I},  K-Theory {\bf 10} (1996),  517-596. 


\bibitem[P-S]{PS}  V. A. Petrov, A.K. Stavrova, {\it Elementary subgroups in isotropic reductive groups},
 Algebra i Analiz {\bf 20} (2008),  160--188; translation in St. Petersburg Math. J. {\bf 20}
 (2009), 625--644.
 
 \bibitem[R]{R} M. S. Raghunathan, {\it Principal bundles admitting a rational section}, Invent. Math. {\bf 116} (1994),
409-423.
 
\bibitem[R-S]{RS} B. Surendranath Reddy,  V. Suresh, {\it 
Admissibility of groups over function fields of p-adic curves}, Adv. Math. {\bf 237} (2013), 316-330. 
 
\bibitem[Sc]{Sc}  P. Schneider, {\it  $p$-adic Lie groups}. Grundlehren der mathematischen Wissenschaften {\bf  344}, Springer, Heidelberg, 2011.

\bibitem[SGA3]{SGA3} {\it S\'eminaire de G\'eom\'etrie alg\'ebrique de l'I.\ H.\ E.\ S., 1963-1964,
sch\'emas en groupes, dirig\'e par M.\ Demazure et A.\ Grothendieck},  
Lecture Notes in Math. 151-153. Springer (1970).




  
  


 \bibitem[St]{St} Stacks project, https://stacks.math.columbia.edu


\bibitem[W]{W}  C. Wenzel,  {\it Rationality of $G/P$ for a nonreduced parabolic subgroup-scheme $P$},
Proceedings of the American Mathematical Society
{\bf 117} (1993),  899-904.



\end{thebibliography}
\end{document}